\documentclass[a4paper,12pt]{amsart}
\usepackage{graphicx}
\usepackage{amsfonts} 
\usepackage{amsmath,amssymb}
\usepackage{tikz-cd}
\usepackage{fullpage}
\newcommand\Pp{\mathbb P}

\newcommand{\cal}{\mathcal}

\newcommand\oo{\mathcal{O}}
\DeclareMathOperator{\bl}{B\ell}
\DeclareMathOperator{\rk}{rk}
\DeclareMathOperator{\jac}{Jac}
\DeclareMathOperator{\Z}{\mathbb Z}
\DeclareMathOperator{\Gr}{Gr}
\title{Derived categories of Gushel--Mukai surfaces and Fano fourfolds of K3 type}

\newtheorem{theorem}{Theorem}
\numberwithin{theorem}{section} 
\newtheorem{lemma}[theorem]{Lemma}
\newtheorem{corollary}[theorem]{Corollary}

\newtheorem{proposition}[theorem]{Proposition}

\theoremstyle{definition}
\newtheorem{definition}[theorem]{Definition}

\theoremstyle{remark}
\newtheorem{remark}[theorem]{Remark}

\theoremstyle{theorem}
\newtheorem{introthm}{Theorem}

\usepackage[colorlinks=true,linkcolor=blue,citecolor=blue,pagebackref
]{hyperref}%
\usepackage[hyperpageref]{backref}

\author[Y. Prieto--Monta\~{n}ez]{Yulieth Prieto--Monta\~{n}ez} %
\address{Pontificia Universidad Católica de Chile, Campus San Joaquín, Avenida Vicuña Mackenna 4860, Santiago de Chile, Chile} %
\email{yulieth.prieto@uc.cl}

\author{Ian Selvaggi} %
\address{Scuola Internazionale Superiore di Studi Avanzati (SISSA), Trieste, Italy} %
\email{iselvagg@sissa.it}

\date{\today}

\begin{document}
\begin{abstract}
    We prove that very general, dual Gushel--Mukai surfaces are not isomorphic, though derived and $L$-equivalent. We use this result to study two semiorthogonal decompositions for a family of Fano fourfolds of K3 type, answering a question by Bernardara--Fatighenti--Manivel--Tanturri.
\end{abstract}
\maketitle
\setcounter{tocdepth}{1}
%\tableofcontents
\section{Introduction}
This article studies the derived categories of some Gushel--Mukai surfaces, building on work by Kuznetsov and Perry \cite{kp2023}, and provides a short account of two families of K3 surfaces, whose members are related by a Fourier--Mukai equivalence but are not isomorphic. As an application, we focus on a family of Fano fourfolds of K3 type defined in \cite{BerFatManTan21}. For a general element in this family, we relate two semiorthogonal decompositions of its bounded derived category by means of homological projective duality.
\subsection{Background}
We start with a recap of some aspects of Gushel--Mukai varieties. For a deeper discussion we refer the reader to \cite{debarre2018gushel}.

Let $V_5$ be a five-dimensional $k$-vector space.%, and let $\Gr(2,V_5)\subset\Pp(\bigwedge^2 V_5)$ be the Pl\"ucker embedding.
\begin{definition}
    A \emph{Gushel--Mukai \textup{(}GM\textup{)} variety} is a smooth $n$-dimensional intersection
\begin{equation*}
X = \mathsf{C}(\Gr(2,V_5)) \cap \Pp^{n+4} \cap Q,
\qquad 
2 \le n \le 6,
\end{equation*}
where $\mathsf{C}(\Gr(2,V_5)) \subset \Pp(\bigwedge^2 V_5\oplus k)$ is the cone over the Grassmannian $\Gr(2,V_5) \subset \Pp(\bigwedge^2 V_5)$ 
in its Pl\"ucker embedding, $\Pp^{n+4} \subset \Pp(\bigwedge^2 V_5\oplus k)$ is a linear subspace, and 
$Q \subset \Pp^{n+4}$ is a quadric hypersurface.
\end{definition}

\begin{remark}
    This definition is slightly less general than the usual one in \cite[Definition 2.1]{debarre2018gushel}, but it is more suited for the purpose of this paper. 
\end{remark}

Note that, in dimension two, Gushel–Mukai surfaces are K3 surfaces, while in other dimensions they are Fano varieties. GM varieties have been extensively studied due to their rich geometry, homological properties (e.g. \cite{kuznetsov2018derived}), and for their connections to Hyperk\"ahler manifolds. In fact, to any GM variety is attached in a canonical way a sextic hypersurface in $\Pp^5$ called a EPW sextic and a corresponding double cover which is a Hyperk\"ahler fourfold. Such a relation was discovered first by O'Grady in \cite{o2006irreducible}, \cite{ogepw}, and further extended by Iliev--Manivel \cite{iliev2011fano}. We briefly recall some of these results. \\

Let $V_6$ be a six-dimensional vector space and let $\bigwedge^3 V_6$ be its exterior power. Consider the natural symplectic form on $\bigwedge^3 V_6$ given by wedge product, and let $A\subset\bigwedge^3 V_6$ be a Lagrangian subspace. For $k\in\Z_{\geq 0}$, the space $\Pp(V_6)$ inherits a natural stratification given by $$Y_A^{\geq k}:=\{v\in \Pp(V_6)\,|\,\dim(A\cap(v\wedge(\bigwedge^2 V_6)))\geq k\}.$$ If $A$ has no decomposable vectors, meaning that $\Pp(A)\cap\Gr(3, V_6)=\emptyset$, then
\begin{itemize}
    \item $Y_A^{\geq 1}$ is a normal irreducible sextic hypersurface, called an EPW sextic,
    \item $Y_A^{\geq 2}$ is a normal irreducible surface,
    \item $Y_A^{\geq 3}$ is finite and reduced, and empty for general $A$,
    \item $Y_A^{\geq 4}=\emptyset$.
\end{itemize}
For a lagrangian subspace $A\subset\bigwedge^3 V_6$, the orthogonal $A^\perp:=\ker(\bigwedge^3 V_6^\vee\rightarrow A^\vee)\subset\bigwedge^3 V_6^\vee$ is again lagrangian and has no decomposable vectors if and only if $A$ does. So, $A^\perp$ induces a dual sequence of closed subvarieties and the loci for $A^\perp$ can be described in terms of $A$ as $$Y_{A^\perp}^{\geq k}:=\{V_5\in \Pp(V_6^\vee)|\,\dim(A\cap\bigwedge^3 V_5)\geq k\}.$$
Among the results of \cite{debarre2018gushel}, a full classification of all isomorphism classes of GM varieties is given. For a Gushel--Mukai variety $X$, there are natural vector spaces associated to it: the six-dimensional vector space $V_6(X)$ of quadrics in $\Pp^{n+4}$ containing $X$, a five dimensional hyperplane $V_5(X)\subset V_6(X)$, and a lagrangian subspace $A(X)\subset\bigwedge^3 V_6(X)$. More specifically, we have the following correspondence.
\begin{theorem}{\cite[Theorem 3.10]{debarre2018gushel}}\label{thmDK}
    For any $n\geq 2$ there are bijections between:
    \begin{itemize}
        \item[(a)] the set of isomorphism classes of GM varieties $X$ of dimension $n\geq 2$ such that $\mathsf C(\Gr(2,V_5))\cap\Pp^{n+4}$ is smooth,
        \item[(b)] the set of pairs $(A,\mathfrak{p})$, where $A\subset\bigwedge^3 V_6$ is a lagrangian subspace with no decomposable vectors and $\mathfrak{p}\in Y_{A^\perp}^{5-n}$, up to the action of $\operatorname{PGL}(V_6)$.
    \end{itemize}
\end{theorem}
The first result of this work concerns isomorphism and $L$-equivalence classes of derived equivalent GM surfaces. Let us first recall the notion of $L$-equivalence. \\
If $k$ is a field, the \emph{Grothendieck group of $k$-varieties} is the free abelian group $K_0(\operatorname{Var}_k)$ generated by isomorphism
classes $[X\rightarrow k]$ of finite type varieties over $k$, modulo the scissor relations, namely the identities 
$[Y]=[X]+[ Y\setminus X]$ whenever $X\hookrightarrow Y$ is a closed subvariety of $Y$. The group $K_0(\operatorname{Var}_k)$ is a ring via $[X]\cdot [Y]=[X\times_kY]$. \\
We denote by $\mathbb L=[\mathbb{A}_k^1]$ the Lefschetz motive, that is the class of the affine line over $k$.
\begin{definition}
    Two smooth, connected, projective varieties $X$ and $Y$ are $L$-equivalent if in $K_0(\operatorname{Var}_k)$ for some $n\geq 0$$$\mathbb{L}^n\cdot([X]-[Y])=0.$$
    We say that they are \textit{nontrivially} $L$-equivalent if in addition $[X]\neq [Y]$.
\end{definition}

We can now formulate our first result.

\begin{introthm}\label{thmYS1}
    Let $A\subset\bigwedge^3 V_6$ be a very general lagrangian subspace such that $$Y_A^{\geq 3}\neq\emptyset\quad\hbox{and}\quad Y_{A^\perp}^{\geq 3}\neq\emptyset.$$
    Let $S_1,\, S_2$ be the Gushel--Mukai surfaces corresponding to points $\mathfrak{p}_1\in Y_A^{\geq 3}$ and $\mathfrak{p}_2\in Y_{A^\perp}^{\geq 3}$. Then $S_1$ and $S_2$ are derived and $L$-equivalent but not isomorphic.
\end{introthm}
To put this theorem into some context, it is worth noting that, among the results of \cite{kuznetsov2018derived}, it is shown that any smooth GM variety $X$ of even dimension contains a K3 category $\cal A_X$. It was conjectured in the same work that if $X_1$ and $X_2$ are generalized duals then $\cal A_{X_1}\cong\cal A_{X_2}$. This has been proved in the subsequent work \cite[Corollary 6.5]{kp2023} and will play a substantial role in our construction. Moreover, the relation between $D$- and $L$-equivalence is still not entirely understood. It was conjectured in \cite{kuznetsov2018grothendieck} that the former implies the latter for simply connected varieties, although this was recently disproved by Meinsma in \cite{meinsma2025counterexamples}. On the other hand, a variation of this conjecture is still open. In the work \cite[Problem 7.2]{ito2020derived}, for a pair $(X,Y)$ of FM partners the equality $[X]=[Y]$ is expected to hold in a quotient of $K_0(\operatorname{Var}_k)[\mathbb{L}^{-1}]$ modulo isogeneous Abelian varieties. A special case has been proved in \cite{caucci2023derived} under further assumptions on the canonical class $K_X$.\\

Our original motivation to delve into this topic was the study of some examples of Fano fourfolds of K3 type. As in the case of GM varieties, their study is particularly interesting due to connections with Hyperk\"ahler geometry.
\begin{definition}
    Let $H=\bigoplus_{p+q=k}H^{p,q}$ be a non-zero Hodge structure. We say that $H$ is of K3 type if $$\max\{q-p\,|\,H^{p,q}\neq 0\}=2\quad\hbox{and}\quad h^{\frac{k+2}{2},\frac{k-2}{2}}=1.$$
\end{definition}
\begin{definition}{\cite[Definition 2.2]{BerFatManTan21}}
    Let $X$ be a smooth Fano variety. We say that $X$ is a Fano variety of K3 type if $H^*(X,\Z)$ contains at least one sub-Hodge structure of K3 type, and at most one for any given weight. 
\end{definition}
An initial source of examples was provided by Küchle in \cite{Kuc95}. Further ones were introduced by Fatighenti and Mongardi in \cite{FatMon21}, where they constructed several families of such varieties. More recently, Bernardara, Fatighenti, Manivel, and Tanturri \cite{BerFatManTan21} gave a comprehensive list of 64 families. The latter classify Fano fourfolds of K3 type which can be constructed as the zero loci of general global sections of homogeneous vector bundles on products of flag manifolds. Some of these families have been studied by J. Hernandez-Gomez in his thesis \cite{hernandezgomez}. We focus on the one called of type \textit{K3-38}. 

\bigskip

A Fano fourfold $X$ is of type K3-38 if it is given by the zero locus in $\Pp^1\times\Pp^1\times\Gr(2,V_5)$ of a general global section in $$H^0 (\Pp^1\times\Pp^1\times\Gr(2,V_5),\oo(1,1,1)\oplus\oo(0,0,1)^{\oplus 3}).$$
Varieties in this family admit two semiorthogonal decompositions for their bounded derived category, both being induced by the projection maps to the factors $\Pp^1\times \Gr(2,V_5)$ and $\Pp^1\times\Pp^1$. Namely, the first one is given by 
$$
D^b(X)=\langle D^b(\Pp^1\times W_5), D^b(S_{20})\rangle,
$$
where $W_5$ is a Fano threefold of index two and degree five, and $S_{20}$ is a K3 surface of degree 20. The existence of such a decomposition is a classical consequence of Orlov's fomula for monoidal transformations \cite{Orlov_1993}. The second one instead consists of 
$$
 D^b(X)=\langle D^b(\Pp^1\times\Pp^1), D^b(\Pp^1\times\Pp^1),D^b(Z)\rangle,
$$
with $Z\rightarrow\Pp^1\times\Pp^1$ a finite flat cover of degree five. It is obtained as an application of a recent result by F. Xie \cite{Xie2021} relating del Pezzo 5 fibrations and semiorthogonal decompositions whose Kuznetsov components are given by derived categories of certain varieties admitting a 5:1 cover of the base. Our second result is a structure theorem for these decompositions and for the varieties involved.
\begin{introthm}\label{thmYS2}
Let $X$ be a very general Fano fourfold of K3-38 type.
    \begin{enumerate}
        \item The two semiorthogonal decompositions above do not coincide.
        \item The surface $Z$ is a K3 surface, derived and $L$-equivalent to $S_{20}$ but not isomorphic.
    \end{enumerate}
\end{introthm}

To our knowledge, the variety $Z$ appearing in the theorem above is the first example known of a K3 surface obtained as Xie's cover. This leaves open studying other families in \cite{BerFatManTan21} where such covers appear.
\subsection*{Structure of the paper} In Section \ref{sectionGM} we first recall the construction of the surfaces appearing in Theorem \ref{thmYS1} and prove their $D$-equivalence. Then, we show that they admit the structure of elliptic surfaces and translate the problem of determining whether they are isomorphic into a problem of studying isomorphism classes of Jacobians of elliptic surfaces

The first part of Section \ref{SectionFK3} is devoted to describe some basic properties of Fano fourfolds of type K3-38 and how the semiorthogonal decompositions mentioned before arise. Then, we relate the surfaces in both SOD's to the ones defined in Section \ref{sectionGM}. This will be enough to conclude Theorem \ref{thmYS2}.
\subsection*{Notation}
We work over an algebraically closed field $k$ of characteristic zero. For the Grassmannian $\Gr(d,V_n)$, we denote by $\mathcal{U}$ the tautological bundle and by $\mathcal{Q}$ the universal quotient bundle. We write $\oo_{\Gr(d,V_n)}(1)$ to indicate the determinant of $\mathcal{Q}$. Lastly, we use the notation $\Pp_B(-)$ and $\Gr_B(k,V_n)$ to indicate a projective bundle and a Grassmann bundle respectively on a base scheme $B$. In accordance with \cite{BerFatManTan21}, if $\cal E$ is a globally generated vector bundle on a variety $X$, we denote $$Y=\cal Z(X,\cal E)$$
for the closed subvariety $Y$ defined as the zero locus of a general global section $\sigma\in H^0(X,\cal E)$. 
\subsubsection*{Acknowledgments}
We warmly thank A. Kuznetsov for suggesting us the some of the key ideas and for his comments on a previous draft of the paper, which greatly improved many aspects of it. We are grateful to M. Bernardara for introducing us to the problem. We thank E. Fatighenti, F. Tufo and S. Secci for many helpful conversations. Lastly, we thank F. Caucci for his help on the literature regarding $D$- and $L$-equivalence. The first author is partially supported by the INICIO-2025-13 grant, titled ``Nuevas direcciones en las variedades hyperkähler y su interacción con las variedades de Fano" ID 250623055. The second author is a member of INDAM-GNSAGA.
\section{Gushel--Mukai surfaces and 
duality}\label{sectionGM}
The purpose of this section is to prove Theorem \ref{thmYS1}. 
We begin by giving a description of the varieties involved.

Let $\sigma_1,\sigma_2,\sigma_3\in H^0(\Gr(2,V_5),\oo(1))$ be general sections; denote by $\Sigma:=\hbox{Span}(\sigma_1,\sigma_2,\sigma_3)\subset\bigwedge^2V_5^\vee$, and $\Sigma^\perp\subset\bigwedge^2V_5$ its 7-dimensional orthogonal complement. Let $W_5:=\Gr(2,V_5)\cap\Pp(\Sigma^\perp)$ the corresponding del Pezzo 3-fold. It is of Picard rank one and degree five (see \cite{fanography}, \href{https://www.fanography.info/1-15}{tag 1-15} for more details). For what follows we will denote $\oo_{W_5}(1)$ for the restriction to $W_5$ of $\oo_{\Gr(2,V_5)}(1)$. 

Let $V_2,\, V_2'$ be 2-dimensional vector spaces, and consider a general section $\sigma\in H^0(\Pp(V_2)\times\Pp(V_2')\times W_5,\oo(1,1,1))$. Let 
\begin{equation}\label{eq4}
V_2\otimes V_2'\rightarrow\bigwedge^2V_5^\vee\bigg/\Sigma
\end{equation}
be the embedding defined by $\sigma$. Indeed, since we have an identification between $H^0(\Pp(V_2)\times\Pp(V_2')\times W_5,\oo(1,1,1))$ and $$ V_2^\vee\otimes V_2'^\vee\otimes\left(\bigwedge^2V_5^\vee\bigg/\Sigma\right)\cong\hbox{Hom}\left(V_2\otimes V_2',\bigwedge^2V_5^\vee\bigg/\Sigma\right),$$ the choice of a general global section corresponds to the embedding in (\ref{eq4}). Lastly, define $D:=(\bigwedge^2V_5^\vee/\Sigma)/(V_2\otimes V_2')$. This gives $\bigwedge^2V_5^\vee$ a filtration whose factors are $\Sigma,\,V_2\otimes V_2',\,D$ of dimensions 3, 4, and 3 respectively.\\
Moreover, by duality there is an induced filtration on $\bigwedge^2V_5$ with factors $D^\vee,\,V_2^\vee\otimes V_2'^\vee,\,\Sigma^\vee$. Given these varieties, there are two associated GM surfaces corresponding via the bijection stated in Theorem \ref{thmDK} to the ones in Theorem \ref{thmYS1}.
\begin{enumerate}
    \item The cone $\mathsf{C}_{\Pp(D^\vee)}(\Pp(V_2^\vee)\times\Pp(V_2'^\vee))$ over the quadric surface $\Pp(V_2^\vee)\times\Pp(V_2'^\vee)$ with vertex $\Pp(D^\vee)$ is a quadric of codimension 4 in $\Pp(\bigwedge^2V_5)$. The intersection 
    \begin{equation}\label{eqGMs}
    S:=\Gr(2,V_5)\cap\mathsf{C}_{\Pp(D^\vee)}(\Pp(V_2^\vee)\times\Pp(V_2'^\vee))
    \end{equation}
    is a K3 surface.
    \item Consider the cone $\mathsf{C}_{\Pp(\Sigma)}(\Pp(V_2)\times\Pp(V_2'))$ over the quadric surface $\Pp(V_2)\times\Pp(V_2')\subset\Pp(V_2\otimes V_2')$ with vertex $\Pp(\Sigma)$. Once again, it is a quadric of codimension 4 in $\Pp(\bigwedge^2V_5^\vee)$, and the intersection 
    \begin{equation}\label{eqGMsdual}
    T:=\Gr(2,V_5^\vee)\cap\mathsf{C}_{\Pp(\Sigma)}(\Pp(V_2)\times\Pp(V_2'))
    \end{equation}
    is another K3 surface.
\end{enumerate}
\begin{definition}{\cite[Definition 3.5]{kuznetsov2018derived}}
    Let $X_1$ and $X_2$ be GM varieties. If there exists an isomorphism $V_6(X_1)\cong V_6(X_2)$ which identifies $A(X_1)\subset\bigwedge^3 V_6(X_1)$ with $A(X_2)^\perp\subset \bigwedge^3 V_6(X_2)^\vee$ then 
    \begin{itemize}
        \item $X_1$ and $X_2$ are \emph{dual} if $\dim(X_1)=\dim(X_2)$ and 
        \item $X_1$ and $X_2$ are \emph{generalized dual} if $\dim(X_1)\equiv\dim(X_2)\mod (2)$.
    \end{itemize}
\end{definition}
By \cite[Proposition 3.28]{debarre2018gushel}, the projective duality between the quadrics in $S$ and $T$ implies that these surfaces are generalized duals. Their derived equivalence then follows from \cite[Corollary 6.5]{kp2023}, which we report below. 
\begin{theorem}{\cite[Corollary 6.5]{kp2023}}
Let $X_1$ and $X_2$ be smooth GM varieties whose associated Lagrangian subspaces $A(X_1)$ and $A(X_2)$ do not contain decomposable vectors. If $X_1$ and $X_2$ are duals, then there is an equivalence $\cal A_{X_1}\cong \cal A_{X_2}$.
\end{theorem}
We are left with proving that $S$ and $T$ are $L$-equivalent and not isomorphic.
\begin{definition}
    An \emph{elliptic surface} is a surface $X$ which admits a surjective morphism $f:X\rightarrow C$ where $C$ is a smooth curve, such that the geometric general fibre is a smooth integral curve of genus one. If $F\in\operatorname{NS}(X)$ is the class of a fibre, the \emph{multisection index} of $f$ is the minimal positive $t>0$ such that there exists a divisor $D\in\operatorname{NS}(X)$ with $D\cdot F=t$.
\end{definition}
\begin{remark}
    When the elliptic fibration considered $f:X\rightarrow C$ is clear from the context, we will drop the dependence of the multisection index on $f$ and just mention the multisection index of $X$
\end{remark}
For what follows, we will say that a surface is of type $S_{20}$ if it is given by $\cal Z(\Pp^1\times W_5,\oo(1,1)^{\oplus2})$. With a little abuse of notation, we will denote such surfaces by $S_{20}$. 

By adjunction formula, $K_{S_{20}}$ is trivial and by the Lefschetz hyperplane theorem the restriction map $H^1(\oo_{\Pp^1\times W_5})\rightarrow H^1(\oo_{S_{20}})$ is an isomorphism. Now, $\Pp^1\times W_5$ is Fano, so $H^1(\oo_{\Pp^1\times W_5})=0$ and $S_{20}$ is a K3 surface. 
\begin{remark}
    As shown in \cite[Corollary 0.3]{mukai1988curves}, a general K3 surface of degree 10 is indeed a complete transverse intersection of a Grassmannian $\Gr(2,V_5)\subset\Pp(\bigwedge^2 V_5)$ with three hyperplanes and a quadric, hence a GM surface. 
\end{remark}
\begin{lemma}\label{lmS20}
         A K3 surface of type $S_{20}$ has Picard rank 2 and degree 20 with respect to the polarization $\oo(1,1)|_{S_{20}}$. Moreover, The restriction to $S_{20}$ of the projection $\pi_1:\Pp^1\times W_5\rightarrow\Pp^1$ defines an elliptic fibration of multisection index 5. 
\end{lemma}
\begin{proof}
    Since $\operatorname{Pic}(\Pp(V_2')\times W_5)\cong \Z^{\oplus 2}$, by Noether--Lefschetz we have that a surface of type $S_{20}$ has Picard rank two. Then, to compute its Néron--Severi lattice it is sufficient to use as a basis for $\operatorname{NS}(S_{20})$ the pullbacks of the two generators of $\operatorname{Pic}(\Pp(V_2')\times W_5)$. 
    Let $H$ and $h$ be the classes of $W_5$ and $\Pp^1$ in $\Pp^1\times W_5$, and denote $H_{S_{20}}$ and $h_{S_{20}}$ their restriction to $S$. Observe that $h$ and $H$ are the classes of two fibres, so $$h^2=0,\quad H^4=0,\quad\hbox{and}\quad H^3\cdot h=5.$$
    Then, the class of $S_{20}$ is equal to $(H+h)^2=H^2+2\,H\cdot h$, so it follows that 
    \begin{align*}
        h_{S_{20}}^2&=H^2\cdot h^2+2\, H\cdot h^3=0\\
        H_{S_{20}}\cdot h_{S_{20}}&=H^3\cdot h+2 H^2\cdot h^2=5\\
        H_{S_{20}}^2&=H^4+2\,H^3\cdot h=10\\
        [S_{20}]^2&=H^4+2\, H^2\cdot h^2+4\,H^3\cdot h=20.
    \end{align*}
    We conclude that 
    \begin{equation}\label{eqNS}
    \operatorname{NS}(S_{20})\cong \begin{pmatrix}
10 & 5 \\
5 & 0
\end{pmatrix},
    \end{equation}
which is isometric to the twisted hyperbolic lattice $U(5)$. To prove the second statement, notice first that the fibre of $\pi_1$ over $s\in\Pp(V_2')$ is given by 
    \begin{equation}\label{eqfib}
    (S_{20})_s=\mathcal{Z}(\Gr(2,V_5),\oo_{\Gr(2,V_5)}(1)^{\oplus 5}). 
    \end{equation} 
    Observe that under our assumptions the sections in $H^0(\Gr(2,V_5),\oo(1))$ defining $S_{20}$  are linearly independent. By generic smoothness, the fibres are smooth over an open subset $U\subset\Pp(V_2')$.  In particular, the closed fibres of $\pi_1$ over $U$ are smooth transverse intersections of $\Gr(2,V_5)$ with a five dimensional linear subspace of $\Pp(\bigwedge^2 V_5)$, hence one-dimensional.  By adjunction formula, we have $$K_{(S_{20})_s}\cong (\oo_{\Gr(2,V_5)}(-5)\otimes\oo_{\Gr(2,V_5)}(5))|_{(S_{20})_s}\cong\oo_{(S_{20})_s},$$
    so they are genus-one curves. Lastly, the description of $\operatorname{NS}(S_{20})$ in \eqref{eqNS} shows that the multisection index of $S_{20}$ is 5.
\end{proof}
\begin{remark}
 The genus one curve appearing in \eqref{eqfib} is called an elliptic     quintic; namely, an elliptic quintic is a smooth projective genus one curve which admits a line bundle of degree ﬁve. Moreover, every elliptic quintic is obtained as the transverse intersection of $\Gr(2,V_5)$ with a codimension 5 subaspace of $\bigwedge^2 V_5$ (see \cite[Lemma 2.6]{ShinderEq20}).
\end{remark}
\bigskip
Let $\varphi:\mathcal{E}\rightarrow\mathcal{F}$ be a morphism of vector bundles over a base scheme $B$, and let $r=\min\{\rk\mathcal{E},\rk\mathcal{F}\}$. We define the $k$-th degeneracy locus of $\varphi$ to be the closed subscheme of $B$ whose ideal locally generated by the $(r+1-k)\times(r+1-k)$ minors of the matrix of $\varphi$. We denote it by $D_k(\varphi).$\\
Consider the projectivization $p:\Pp_B(\cal E)\rightarrow B$. By \cite[Lemma 2.1]{kuznetsov2016kuchle}, $\varphi$ gives a global section of the vector bundle $p^*\cal F\otimes\oo_{\Pp_B(\cal E)}(1)$.
\begin{proposition}\label{propS20}
 The surface $S$ in \eqref{eqGMs} is isomorphic to a surface of type $S_{20}$.
\end{proposition}
\begin{proof}
As mentioned before, by Lemma 2.1 in \cite{kuznetsov2016kuchle}, the choice of a global section $$\varphi\in H^0(\Pp(V_2'^\vee)\times W_5,V_2^\vee\otimes\oo(1,1))=H^0(\Pp_{W_5}(V_2'^\vee\otimes\oo_{W_5}),V_2^\vee\otimes\oo(1,1))$$
corresponds to the morphism of vector bundles over $W_5$
    $$\varphi:V_2'^{\vee}\otimes\oo_{W_5}\rightarrow V_2^\vee\otimes\oo_{W_5}(1).$$
Next, observe that if the section $\varphi$ is chosen to be general, the second degeneracy locus $D_2(\varphi)$ of $\varphi$ is empty for dimensional reasons. Indeed, with such hypotheses, asking for the morphism $\varphi$ to have rank 0 corresponds to imposing four independent linear conditions on the threefold $W_5$. As a result, we conclude by \cite[Lemma 3.2]{fatighenti2024geometry} that there is an isomorphism between the first degeneracy locus, whose ideal sheaf is locally generated by $(\det(\varphi))$ and a surface of type $S_{20}$. Namely, we have an isomorphism between $$D_1(\varphi)=\Gr(2,V_5)\cap\mathsf{C_{\Pp(D^\vee)}}(\Pp(V_2^\vee)\times\Pp(V_2'^\vee))=S,$$ and $S_{20}=\mathcal{Z}(\Pp(V_2'^\vee)\times W_5,V_2^\vee\otimes\oo(1,1))\subset\Pp_{W_5}(V_2'^\vee\otimes\oo_{W_5})=\Pp(V_2'^\vee)\times W_5$. 
\end{proof}
\begin{remark}
    In principle, in the proof of the last Lemma one needs to be a bit careful. In fact, both \cite[Lemma 2.1]{kuznetsov2016kuchle} and \cite[Lemma 3.2]{fatighenti2024geometry}, are stated for a morphism of vector bundles $\varphi:\mathcal{E}\rightarrow\mathcal{F}$ for which $\rk\mathcal{E}\geq\rk\mathcal{F}+1$. However, the same proof applies if $\rk\mathcal{E}=\rk\mathcal{F}=2$.
\end{remark}
Arguing as in the previous Proposition, it can be showed that $$T\cong \cal Z(\Pp^1\times(\Gr(2,V_5^\vee)\cap \Pp((D^\vee)^\perp)),\oo(1,1)^{\oplus 2}),$$
and so, $T$ comes as an elliptic K3 surface whose fiber over a general point $s\in\Pp^1$ is given by
\begin{equation}\label{eqfibdual}
T_s=\cal Z\left(\Gr(2,V_5^\vee),\oo_{\Gr(2,V_5^\vee)}(1)^{\oplus 5}\right).
\end{equation}
\subsubsection{Elliptic K3 surfaces of Picard rank 2 and Jacobians}
\begin{proposition}{\cite[Remark 4.2]{van2005some}\cite[Proposition 3.1]{MeiShi24}}
    Let $X$ be an elliptic K3 surface of Picard rank 2, and let $F\in\operatorname{NS}(X)$ be the class of a fibre. Then there exists a polarization $H$ on $X$ such that $H,\,F$ form a basis of $\operatorname{NS}(X)$ and $H\cdot F=t$. The Néron--Severi lattice of $X$ is given by a matrix of the form $$\operatorname{NS}(X)\cong \begin{pmatrix}
2d & t \\
t & 0
\end{pmatrix}$$
\end{proposition}
Denote $\Lambda_{d,t}$ for the lattice of rank 2 given by the matrix above. It has exactly two isotropic primitive vectors up to sign: the class $F$ and $$F'=\frac{1}{\gcd(d,t)}(tH-dF).$$
In many cases, the class $F'$ gives another elliptic fibration.
\begin{lemma}{\cite[Section 4.7]{van2005some}}
    Let $X$ be a very general K3 surface of Picard rank 2, and Néron--Severi lattice $\Lambda_{d,t}$. If $t>2$ and $d\not\equiv-1\mod t$ then $F'$ gives a second elliptic fibration which is isomorphic (as an elliptic surface) to that defined by $F$ if and only if $d\equiv 1\mod t$.
\end{lemma}
Given any smooth projective curve $C$, its $k$-th Jacobian $\operatorname{Jac}^k(C)$ is defined as the moduli space of degree $k$ line bundles on $C$. A similar construction applies to elliptic surfaces, motivating the following. 
\begin{definition}
Let $X\rightarrow C$ be an elliptic surface. For every $k\in\mathbb Z$, the $k$-th Jacobian of $X$, denoted $\operatorname{Jac}^k(X/C)$, is another elliptic surface over $C$ defined as the unique minimal regular model with the generic fibre given by $\operatorname{Jac}^k(X_{k(C)})$.
\end{definition}
For a full treatment of Jacobians of elliptic surfaces we refer to \cite[Chapter 2]{Dolgachev2011}.\\
Let $t$ be the multisection index of $X$. The Jacobians for which $\gcd(d,t)=1$ are called \emph{coprime Jacobians} of $X$. The state of the art regarding Fourier--Mukai partners of elliptic K3 surfaces is the following result by Meinsma and Shinder.
\begin{theorem}{\cite[Theorem 1.2]{MeiShi24}}\label{thmms}
Let $X$ be an elliptic K3 surface of Picard rank 2. Let $t$ be the multisection index of $X$ and let $2d$ be the degree of a polarization on $X$. Denote $m=\operatorname{gcd}(d,t)$. 
\begin{enumerate}
    \item[(i)] If $m=1$, then 
    every Fourier--Mukai partner of $X$
    is isomorphic to a coprime Jacobian
    of a fixed elliptic fibration on $X$;
    
    \item[(ii)] If $m=p^k$, for a prime $p$, then every Fourier--Mukai
    partner of $X$ is isomorphic
    to a coprime Jacobian
    of one of the two elliptic fibrations on $X$;
    
    \item[(iii)] If $m$ is not a power of a prime, and $X$ is very general, then $X$ admits Fourier--Mukai partners which are not isomorphic to any Jacobian of any elliptic fibration on $X$.
\end{enumerate}
\end{theorem}
\begin{proposition}{\cite[Proposition 3.10(3)]{ShinderEq20}}\label{propshinder0}
    Let $S$ be an elliptic K3 surface of Picard rank 2 and $\operatorname{NS}(S)\cong U(5) $, and let $S':=\jac^2(S/\Pp^1)$. If $S$ is very general, then $S\not\cong S'$.
\end{proposition}
Though not immediately related, the following lemma will be necessary later.
\begin{lemma}{\cite[Proposition 2.8]{ShinderEq20}}\label{lmshinder}
        Let $V_5$ be a 5-dimensional vector space and let $L\subset\bigwedge^2(V_5^\vee)$ be a 5-dimensional subspace.Let \begin{align}
         C&:=\Gr(2,V_5)\cap\Pp(L^\perp)\label{curve}\\
         C'&:=\Gr(2,V_5^\vee)\cap\Pp(L). \label{dualcurve}
         \end{align}
         If $C$ is a smooth transverse intersection so is $C'$, and we have 
         \begin{equation}\label{eqcurve}
         C'\cong\jac^3(C),\quad C\cong\jac^2(C').
         \end{equation}
\end{lemma}
We are ready for the main result of the section.
\begin{theorem}
If the K3 surfaces $S$ and $T$ are very general then they are not isomorphic, though derived and $L$-equivalent.
\end{theorem}
\begin{proof}
It is enough to prove that $T\cong\jac^2(S/\Pp^1)$, where the elliptic fibrations considered on $S$ ant $T$ are those given by \eqref{eqfib} and \eqref{eqfibdual} respectively. This, together with the very generality assumption, would imply that $S\not\cong T$ by Proposition \ref{propshinder0}. Observe that in our setting, the curve in \eqref{eqfib} in Lemma \ref{lmshinder} coincides with \eqref{curve}, and the analogue holds for the elliptic fibres in $T$ and \eqref{dualcurve}. This means that, for general $s\in\Pp^1$, $T_s\cong\jac^2(S_s)$. This construction can be carried over generic points. Indeed, let $\eta\in\Pp^1$ be the generic point. Then the fibre $S_\eta$ is the zero locus of a global section of $\oo(1)^{\oplus 2}$ on $W_5\times_k \operatorname{Spec} (k(\Pp^1))$, that is the transverse intersection $$S_\eta=(\Gr(2,V_5)\times_k \operatorname{Spec}k(\Pp^1))\cap\Pp(L^\perp)\subset\Pp\left(\bigwedge^2V_5^\vee\right)\otimes_k \operatorname{Spec}k(\Pp^1),$$ 
where $L$ is a five-dimensional subspace in $\bigwedge^2 V_5^\vee\otimes_k k(\Pp^1)$. Analogously, we see that $T_\eta$ is the intersection $T_\eta=(\Gr(2,V_5^\vee)\times_k \operatorname{Spec}k(\Pp^1))\cap\Pp(L)\subset\Pp\left(\bigwedge^2V_5^\vee\right)\otimes_k \operatorname{Spec}k(\Pp^1)$. In particular, Lemma \ref{lmshinder} applies and we conclude that $$T_\eta\cong\jac^2(S_\eta).$$
This in turn implies that $T\cong\jac^2(S/\Pp^1)$. Lastly, by \cite[Theorem 3.2]{ShinderEq20} we have $$\mathbb{L}^4([S]-[T])=0,$$
in $K_0(\operatorname{Var}_k)$.
\end{proof}
\section{On the geometry of the Fano K3-38}\label{SectionFK3}
As mentioned in the introduction, in this section we turn our focus to a particular family a Fano fourfolds. We give a brief description of their geometry and eventually prove Theorem \ref{thmYS2}.

As in the previous section, denote by $W_5:=\mathcal{Z}(\Gr(2,V_5),\oo(1)^{\oplus3})$ the del Pezzo threefold of degree five obtained as a section of Plücker embedding of $\Gr(2,V_5)$ by a codimension 3 subspace. 
\begin{definition}
    A Fano fourfold $X$ is of type K3-38 if it is a general divisor of type $(1,1,1)$ in the product $\Pp^1\times\Pp^1\times W_5.$
\end{definition}
\begin{remark}
    In the article \cite{BerFatManTan21} such fourfolds were defined as the zero loci $$\mathcal{Z}\left(\Pp^1\times\Pp^1\times\Gr(2,V_5),\oo(1,1,1)\oplus\oo(0,0,1)^{\oplus3}\right).$$
    However, it is clear that the two definitions are equivalent.
\end{remark}
For clarity, we prove the following description as stated in \cite{BerFatManTan21}.
\begin{proposition}\label{prop:geometry X} Let $X$ be a Fano fourfold of type K3-38, let $\pi_{12}: X \rightarrow \mathbb{P}^1 \times \mathbb{P}^1$ and $\pi_{13},\,\pi_{23}:X\rightarrow\Pp^1\times W_5$ be the projection maps. Then:
\begin{itemize}
    \item[(i)] the map $\pi_{12}$ is a flat, proper morphism whose fibres are del Pezzo surfaces of degree 5 with rational Gorenstein singularities;
    \item[(ii)] $\pi_{13},\pi_{23}$ are blow up maps $\bl_{S_{20}}(\Pp^1\times W_5)\rightarrow\Pp^1\times W_5$ where $S_{20}$ is a K3 surface of degree 20 and Picard rank 2.
\end{itemize}
In particular, $X$ is rational.
\end{proposition}
\begin{proof}
    Consider the projection $\pi_{12}:X\rightarrow\Pp^1\times\Pp^1$. The fibre $X_s$ of $\pi_{12}$ over a closed point $s\in \Pp^1\times\Pp^1$ is given by $W_5$ cut by a section $\sigma_s\in H^0(W_5,\oo(1))$. For what follows, denote by $\sigma_1,\,\sigma_2\,,\,\sigma_3\in H^0(\Gr(2,V_5),\oo(1))$ the sections defining $W_5$. The Fano $X$ is assumed to be general in the linear series $|\oo(1,1,1)|$. This implies that
    \begin{itemize}
        \item[(a)] for each $s\in\Pp^1\times\Pp^1$ the sections $\sigma_s,\,\sigma_1|_{W_5},\,\sigma_2|_{W_5},\,\sigma_3|_{W_5}$ are linearly independent, thus $X_s$ is two-dimensional,
        \item[(b)] the singular fibres arise from those hyperplanes $H_{\sigma_t}\subset\Pp(\wedge^2 V_5)$ which are tangent to $W_5$, or correspond to degenerate 2-forms on $V_5$ of rank 2. In both cases the singularities are of type $A_1$.
    \end{itemize}
    By item (a) the morphism $\pi_{12}$ is flat and the singularities are Gorenstein. Item (b) implies then that the singularities are also rational. Hence the adjunction formula gives $$K_{X_s}\cong (K_{\Gr(2,V_5)}\otimes\oo_{\Gr(2,V_5)}(4))|_{X_s}\cong \oo_{\Gr(2,V_5)}(-1)|_{X_s}.$$
    Thus, the anticanonical bundle $-K_{X_s}$ is ample, and for each $s\in\Pp^1$ fibre of $\pi_{12}$ is a del Pezzo surface with rational Gorenstein singularities. Then, it is easy to compute the degree of $X_s$, as it is given by four linear sections and $\deg X_s=\deg\Gr(2,V_5)=5$. 
    \\
    To tackle the second item, we work with $\pi_{23}:X\rightarrow\Pp^1\times W_5$, the argument for $\pi_{13}$ being analogous. We use the ``Cayley trick" in \cite[Lemma 3.1]{BerFatManTan21}, which is a special case of \cite[Lemma 2.1]{kuznetsov2016kuchle}. Consider the vector bundle $\mathcal{L}:=\oo(1,1)$ on $Y:=\Pp^1\times W_5$. We can write $X$ as $$X=\mathcal{Z}(\Pp^1\times Y,\oo_{\Pp^1}(1)\boxtimes\mathcal{L}),$$
    so that by the Cayley trick $X\cong \bl_{S_{20}} Y$, where $S_{20}=\mathcal{Z}(Y,\mathcal{L}^{\oplus2})=\mathcal{Z}(\Pp^1\times W_5,\oo(1,1)^{\oplus2})$. The rest of the properties for $S_{20}$ follow from Lemma \ref{lmS20}. Finally, since $W_5$ is rational by \cite[Theorem 4.2]{iskovskih1978fano}, the identity $X\cong\bl_{S_{20}}(\Pp^1\times W_5)$ provides a rational model for $X.$
\end{proof}
\begin{definition}
    Let $f:X\rightarrow S$ be a morphism of smooth algebraic varieties. A triangulated subcategory $\cal A\subset D^b(X)$ is called $S$-linear if $A\otimes^L Lf^*F\in\cal A$ for any $A\in\cal A$ and $F\in D^b(S)$.
\end{definition}
\begin{proposition}
    For a Fano fourfold of K3-38 type there are two semiorthogonal decompositions of its bounded derived category, namely
    \begin{align}
        D^b(X)&=\langle D^b(\Pp^1\times W_5), D^b(S_{20})\rangle\label{SODbl}\\
        D^b(X)&=\langle D^b(\Pp^1\times\Pp^1), D^b(\Pp^1\times\Pp^1),D^b(Z)\rangle,\label{SODxie}
    \end{align}
    with $Z\rightarrow\Pp^1\times\Pp^1$ flat and finite of degree 5. Moreover, they are linear over $\Pp^1\times W_5$ and $\Pp^1\times\Pp^1$ respectively.
\end{proposition}
\begin{proof}
    The existence of the first semiorthogonal decompotition is a standard fact. Indeed, by Proposition \ref{prop:geometry X}(ii) the maps $\pi_{13}$ and $\pi_{23}$ are blow up maps $\bl_{S_{20}}(\Pp^1\times W_5)\rightarrow\Pp^1\times W_5$. It suffices then to apply Orlov's formula for a blow up (\cite[Theorem 4.3]{Orlov_1993}) to get \eqref{SODbl} together with $(\Pp^1\times W_5)$-linearity. \\
    The existence of the SOD in \eqref{SODxie} follows from Proposition \ref{prop:geometry X}(i) and \cite[Theorem 1.1]{Xie2021}. Namely, if $\mathcal{X}\rightarrow B$ is a flat morphism where each fibre is a Del-Pezzo surface of degree 5 with rational Gorenstein singularities, there is a $B$-linear semiorthogonal decomposition $$D^b(\mathcal{X})=\langle D^b(B),D^b(B), D^b(\mathcal{Z})\rangle,$$
    where $\mathcal{Z}\rightarrow B$ is flat and finite of degree 5. This in our case reads exactly as \eqref{SODxie}.
\end{proof}
\begin{remark}
    The fact that the Kuznetsov component in \eqref{SODbl} is a K3 category can be seen as an instance of a more general result. Let $M$ be a smooth projective fivefold with a fixed polarization $H$ such that $-K_M=2H$. Suppose $M$ admits a rectangular Lefschetz decomposition $$D^b(M)=\langle\mathcal{A,\mathcal{A}}(H)\rangle.$$
    Let $Y\in|H|$ be a general divisor and let $i:Y\hookrightarrow M$ be the corresponding embedding. The methods in \cite{Kuznetsov+2019+239+267} provide a semiorthogonal decomposition $$D^b(Y)=\langle\mathcal{A}_Y,\mathcal{R}_Y\rangle,$$
    where $\mathcal{A}_Y$ is the image of $\mathcal{A}$ under the functor $i^*:D^b(M)\rightarrow D^b(Y)$ and $\mathcal{R}_Y$ is a K3 category. For a Fano of type K3-38 $X$ this is easy to see. \\
    
    Let $N:=\Pp^1\times W_5$, and consider the fivefold $M:=\Pp^1\times N$. Observe also that $-K_M=2 H$, where $H=\oo_{\Pp^1\times\Pp^1\times\Gr(2,V_5)}(1,1,1)|_{M}$. The projective bundle formula for the projection $M\rightarrow N$ (\cite[Theorem 2.4]{kuznetsov2014semiorthogonal}) gives $$D^b(M)=\langle D^b(N), D^b(N)\otimes\oo(H)\rangle.$$
    Taking $\cal A:= D^b(N)$, this induces a semiorthogonal decomposition $D^b(X)=\langle\mathcal{A}_X,\mathcal{R}_X\rangle$, with $\mathcal{R}_X$ a K3 category. 
\end{remark}
In the paper \cite{BerFatManTan21}, after briefly describing the exceptional objects involved, the authors ask whether the decompositions \eqref{SODbl} and \eqref{SODxie} coincide, and whether the variety $Z$ in \eqref{SODxie} is a K3 surface, derived equivalent to $S_{20}$.\\
Using the constructions in section \ref{sectionGM}, we are able to give an answer.
\begin{proposition}
    The semiorthogonal decompositions \eqref{SODbl} and \eqref{SODxie} do not coincide.
\end{proposition}
\begin{proof}
    Looking at the Kutnezsov components of both decompositions we see that $D^b(S_{20})$ is $(\Pp^1\times W_5)$-linear, while $D^b(Z)$ is $(\Pp^1\times\Pp^1)$-linear. In general, for a variety $M$, if a semiorthogonal component of $D^b(M)$ is $M$-linear then it is zero or the whole $D^b(M)$. In particular, if these decompositions coincide then the Kuznetsov components would be $(\Pp^1\times\Pp^1\times W_5)$-linear, which contradicts them not being zero or $D^b(\Pp^1\times\Pp^1\times W_5)$. 
\end{proof}

The next goal is to prove that the surface $Z$ appearing in \eqref{SODxie} is isomorphic to the surface $T$ appearing in \eqref{eqGMsdual}. If $X$ is very general, this gives immediately Theorem \ref{thmYS2}.

For what follows we recall a few results from \cite{Xie2021}. Let $\mathcal{X}\xrightarrow{\pi}B$ be a quintic del--Pezzo fibration, that is a flat morphism such that for any point $b\in B$ the fibre $\cal X_b$ is a del Pezzo surface of degree five with rational Gorenstein singularities.\\
Let 
\begin{align*}
    h_2(t)&:=5(t+1)^2,\\
    h_3(t)&:=\frac{1}{2}(t+1)(5t+6),
\end{align*}
be numerical polynomials. Consider the relative moduli spaces $\mathcal{M}_i(\cal X/B)$, $i\in\{2,3\}$, of semistable sheaves on the fibers of $\cal X\rightarrow B$ with Hilbert polynomial $h_i(t)$.
\begin{lemma}{\cite[Proposition 7.3]{Xie2021}}\label{lmXie}
For $i\in\{2, 3\}$, $\mathcal{M}_i(\cal X/B)$ are ﬁne moduli spaces. Let $\mathcal{E}_i$ be the universal families of $\mathcal{M}_i(\cal X/B)$. Then 
\begin{itemize}
  \item[(i)]
$\mathcal{M}_2(\cal X/B)\cong B$ and $\mathcal{E}_2|_{\cal X_s}$
is a vector bundle of rank 2. 
\item[(ii)] $\mathcal{M}_3(\cal X/B)\cong Z\rightarrow B$ is flat and finite of degree 5. 
Moreover, $\mathcal{E}_3$ is ﬂat over $\mathcal{M}_3(\cal X/B)$ and is a locally free sheaf over $\cal X$ of rank 1. 
\end{itemize}
Also, $\mathcal{E}_i$ has finite $\operatorname{Tor}$-amplitude over $\mathcal{M}_i(\cal X/B)$ and finite $\operatorname{Ext}$-amplitude over $\cal X$.
\end{lemma}
As explained in \cite{Xie2021}, the universal family $\cal E_2$ of the moduli space $\cal M_2(\cal X/B)$ induces the cartesian square
$$\begin{tikzcd}
\cal X\arrow[r]\arrow[d]&\Gr_B(2,(\pi_*\cal E_2)^\vee)\arrow[d]\\
\Pp_B((\pi_*\omega_{\cal X/B}^{-1})^\vee)\arrow[r]&\Pp_B(\bigwedge^2 (\pi_*\cal E_2)^\vee),
\end{tikzcd}
$$
where $\omega_{\cal X/B}^{-1}$ is the relative anticanonical sheaf (see \cite[Section 7.2]{Xie2021} for more details). Define $\cal F^\perp$ to be the kernel $$0\rightarrow\cal F^\perp\rightarrow\bigwedge^2 \pi_*\cal E_2\rightarrow \pi_*\omega_{\cal X/B}^{-1}\rightarrow 0,$$
and $Z'$ the corresponding dual linear section, i.e. 
$$Z':=\Pp_B(\cal F^\perp)\times_B\Gr_B(2,\pi_*\cal E_2)\subset\Pp_B\left(\bigwedge^2 \pi_*\cal E_2\right).$$
\begin{theorem}{\cite[Theorem 7.4]{Xie2021}}\label{thmxie}
    Let $\pi:\cal X\rightarrow B$ be a quintic del Pezzo fibration. Denote $\pi':Z'\rightarrow B$ the fibration above. Then there is a $B$-linear semiorthogonal decomposition $$D^b(\cal X)=\langle D^b(B), D^b(B), D^b(Z')\rangle.$$
    In particular, $Z'$ is isomorphic to the variety $Z$ appearing in Lemma \ref{lmXie}(ii).
\end{theorem}
We are ready for the main result of this section.\\
Recall the notation introduced in section \ref{sectionGM}. If $\sigma_1,\sigma_2,\sigma_3\in H^0(\Gr(2,V_5),\oo(1))$ are general sections, we denote $\Sigma=\hbox{Span}(\sigma_1,\sigma_2,\sigma_3)\subset\bigwedge^2V_5^\vee$, and $\Sigma^\perp\subset\bigwedge^2V_5$ its 7-dimensional orthogonal complement. In particular, $W_5=\Gr(2,V_5)\cap\Pp(\Sigma^\perp)$. Let $V_2$ and $V_2'$ be 2-dimensional vector spaces, so that $\Pp^1\times\Pp^1=\Pp(V_2)\times\Pp(V_2')$.
\begin{proposition}
    Let $X$ be a Fano fourfold of type K3-38, and let $Z$ and $T$ be the surfaces defined by \eqref{SODxie} and \eqref{eqGMsdual} respectively. There is an isomorphism $Z\cong T$.
\end{proposition}
\begin{proof}
    The homological projective dual of $\Gr(2,V_5)\hookrightarrow\Pp(\bigwedge^2V_5)$ is given by $\Gr(2,V_5^\vee)\hookrightarrow\Pp(\bigwedge^2V_5^\vee)$. So, by \cite[Proposition A.10]{kp2023}, the homological projective dual of $W_5\subset\Pp(\Sigma^\perp)$ is $$\Gr(2,V_5^\vee)\rightarrow\Pp\left(\bigwedge^2V_5^\vee\bigg/\Sigma\right),$$
        which is a 5:1 cover, since the Plücker embedding has degree 5. If we consider the base change of this cover along the map $\Pp(V_2)\times\Pp(V_2')\rightarrow\Pp(\bigwedge^2V_5^\vee/\Sigma)$ induced by the section $\sigma\in H^0(\Pp(V_2)\times\Pp(V_2')\times W_5,\oo(1,1,1))$ defining $X$, we obtain the cover in Theorem \ref{thmxie}. So we have the following cartesian diagram 
        $$
        \begin{tikzcd}
            Z\arrow[r]\arrow[d]&\Gr(2,V_5^\vee)\arrow[d]\\
\Pp(V_2)\times\Pp(V_2')\arrow[r,hook]&\Pp(\bigwedge^2V_5^\vee/\Sigma),
        \end{tikzcd}
        $$
        and taking this pullback is the same as taking the intersection $\Gr(2,V_5^\vee)\cap\mathsf{C}_{\Pp(\Sigma)}(\Pp(V_2)\times\Pp(V_2')).$ In particular, $Z\cong T$.
\end{proof}
Combining the above results, we have shown the following.
\begin{corollary}
    Let $X$ be a very general Fano fourfold of type K3-38. The varieties $S_{20}$ and $Z$, whose bounded derived categories appear as the Kuznetsov components of the semiorthogonal decompositions \eqref{SODbl} and \eqref{SODxie} respectively, are K3 surfaces, which are non isomorphic, $L$-equivalent, Fourier--Mukai partners.
\end{corollary}
\bibliographystyle{alpha}
\bibliography{references}
\end{document}